\documentclass[11pt]{amsart}
\usepackage{geometry}                % See geometry.pdf to learn the layout options. There are lots.
\usepackage{graphicx,amssymb,amscd,amsxtra,amsmath,amsfonts,amsthm,footnpag}
\usepackage[labelfont=bf]{caption}
\DeclareCaptionLabelSeparator{period}{.  }
\usepackage{epstopdf}

\renewcommand{\thefootnote}%
{\fnsymbol{footnote}}

\theoremstyle{definition}

\theoremstyle{plain}
\newtheorem{Thm}
{Theorem}
\newtheorem{Lemm}[Thm]{Lemma}
\newtheorem{Cor}[Thm]{Corollary}
\newtheorem{Propn}[Thm]{Proposition}

\newtheorem*{Thm*}{Theorem}
\newtheorem*{Lemm*}{Lemma}
\newtheorem*{Propn*}{Proposition}
\newtheorem*{Cor*}{Corollary}

\theoremstyle{remark}
\newtheorem*{Rmk}{\bf Remark}
\newcommand{\abs}[1]{\lvert#1\rvert}

\DeclareMathSymbol
{\rightrightarrows}
{\mathrel}{AMSa}{"13}

\title{Periodic Points on the $2$-sphere}
\author{Charles Pugh and Michael Shub}
\thanks{Michael Shub was supported by CONICET PIP 0801 2010-2012  and ANPCyT PICT 2010-00681.}

\address{Charles Pugh \\ Department of Mathematics \\University of Chicago\\ 5734 S. University Ave.\\
Chicago, Illinois 60637\\ USA \\ pugh@math.uchicago.edu {\em and  }
Department of Mathematics\\ University of California, Berkeley\\ 970 Evans Hall $\#$3840\\Berkeley, CA 94720 USA pugh@math.berkeley.edu}
\address{Michael Shub \\ CONICET, IMAS, Universidad de Buenos Aires, Buenos Aires, Argentina \\\emph{and} Department of Mathematics \\ The CUNY Graduate Center\\ 365 Fifth Avenue, Room 4208 \\ New York, NY 10016 \\  shub.michael@gmail.com}

\dedicatory{To Arieh Iserles on his 65$^{th}$ birthday}

\begin{document}
\maketitle
%\section{}
%\subsection{}

\begin{abstract}
For a $C^{1}$ degree two latitude preserving endomorphism $f$ of the $2$-sphere, we show that $f$ has $2^{n}$ periodic points. 
\end{abstract}

\section{Introduction}
The relationship between the long term dynamics of an endomorphism of a manifold and its long term effect on the algebraic topology of the manifold can depend on the smoothness of the endomorphism.  See \cite{Spain} for a discussion of this.  Here we deal with a particular case of  Problem~3 of that paper.  

Let  $S$ be the $2$-sphere, oriented in the standard fashion.  Fix a   continuous map $f : S \rightarrow S$ of global degree $2$.  That is, the map $f_{_{\displaystyle *}} : H_{2}(S) \rightarrow H_{2}(S)$ is multiplication by $2$.  Problem~3 asks:  Is the Growth Rate Inequality 
$$
 \limsup_{n\rightarrow \infty} \frac{1}{n}\ln N_{n}(f)  \geq   \ln(2)
$$
true?  Here $N_{n}(f)$ is the number distinct periodic points of $f$ having period $n$, i.e., the number of fixed points of $f^{n}$.  

If $f$ is merely continuous   the answer is ``not necessarily.''  For as observed in \cite{SS}, a Lipschitz counterexample  is $z \mapsto 2z^{2}/\abs{z}$ 
 where $S$ is the Riemann sphere.  In polar coordinates $f$ sends $(r, \theta )$ to $(2r, 2\theta )$.  The only periodic points are  the North and South poles. 
 
 On the other hand, if $f$ is a rational map, then the answer is ``yes.''  See Proposition~1 of \cite{Asterisque}.  So the question becomes: Does there exist an $r \geq  1$ such that if $f$ is $C^{r}$ then the Growth Rate Inequality holds?    Perhaps $r=2$ will suffice, or even $r=1$.
 
 From \cite{SS} it follows that if $f$ is $C^{1}$ then it must have infinitely many distinct periodic points, but their growth rate remains unknown.  
  As shown in \cite{MP}, the topological entropy of $f$ is better understood:  If $f$ is $C^{1}$ then it is at least $\ln 2$.  This implies there are invariant probability measures with measure theoretic entropy $\ln 2 - \epsilon $.  Thus, if $f$ is $C^{1+ \alpha }$ then the sum of the Lyapunov exponents is at least $\ln 2 - \epsilon $.  
 
 We expect a version of Katok's Theorem \cite{Katok} about   diffeomorphisms to be true for endomorphisms in the case that  all the Lyapunov exponents are different from zero.  In fact this is already proved in the case that all the Lyapunov exponents are positive.  See \cite{GW}.  The remaining case will be where one of the Lyapunov exponents is zero.
 
 At the end of Section~\ref{s.examples} we give three examples of smooth endomorphisms of the $2$-sphere, two  with topological entropy $\ln  2$, the other with topological entropy $\ln 3$.  All have one Lyapunov exponent zero and are essentially $2:1$.   The first is of degree zero and  has only one periodic point.  The second and third have degree two, are unlike the map $z \mapsto z^{2}$, but satisfy the Growth  Rate Inequality.  All three examples preserve the latitude foliation.  
 
 We hope this elementary result might give a clue as to how   homology assumptions can intervene when there are zero Lyapunov exponents and also   families of invariant  center manifolds replacing the circles of our foliation, but that will surely go way beyond what we can accomplish here.

\section{Invariant latitudes}

 We say that $f : S \rightarrow  S$   preserves the latitude foliation if it carries each latitude into another  latitude or  to one of the poles.  It need not be a homeomorphism from one latitude to another.  We assume throughout that $f : S \rightarrow  S$ is a continuous, latitude preserving endomorphism of degree $2$.  
 
 \begin{Thm}
 \label{t.fixedpoints}
If $f$ is $C^{1}$  then $f^{n}$ has at least $2^{n}  $ fixed points. 
  \end{Thm}
  
  \begin{Cor}
\label{c.GRF}
If  $f$ is $C^{1}$ and 
  $N_{n}$ is the number of fixed points of $f^{n}$ then the Growth Rate Inequality
  $$
  \limsup_{n \rightarrow \infty} \frac{1}{n} \ln (N_{n}) \geq \ln(2).
  $$
  holds.
\end{Cor}

\begin{proof}[\bf Proof]
This is immediate from the theorem, but see the remark at the end of Section~\ref{s.counting} for a shorter proof of the corollary.
\end{proof}

\begin{Rmk}  
$f^{n}$ is the $n^{\textrm{th}}$ iterate of $f$ and the fixed points are geometrically distinct.  The $C^{1}$ assumption is used rarely in the proof, but without it the theorem fails:  As  noted above, the Lipschitz endomorphism $z \mapsto 2z^{2}/\abs{z}$ from \cite{SS} has degree  $2$ but only two periodic points.   
\end{Rmk}

\begin{Lemm}
\label{l.poletopole} A latitude preserving endomorphism  sends a pole to a pole, not a latitude.
\end{Lemm}

\begin{proof}[\bf Proof]
Obvious from  continuity of the endomorphism.
\end{proof}

Parametrize the latitudes by their height $h$, $0 \leq  h \leq 1$.  (The sphere $S$ rests on the $xy$-plane and has center $(0,0,1/2)$.)  The Southpole corresponds to $h=0$ and the Northpole to $h=1$.  If $L(h)$ denotes the latitude of height $h$ then we define the \textbf{latitude map} $\varphi  : [0,1] \rightarrow [0,1]$ by 
$$
\varphi (h) = \textrm{ height of } f(L(h)).
$$
It is continuous and Lemma~\ref{l.poletopole} implies that $\varphi \{0,1\} \subset \{0,1\}$.  The map $f$ fibers over $\varphi $.  
 
Orient each latitude circle in $S$ in a counterclockwise fashion as viewed from above the Northpole of $S$.  If $0 < h, \varphi (h) < 1$ then the \textbf{latitude degree} $d(h)$ is the degree of the map $f : L(h) \rightarrow L(\varphi (h))$.  Since these latitude circles are oriented, $d(h)$ is well defined and locally constant as a function of $h \in ((0,1) \cap \varphi ^{-1}(0,1))$.

Corresponding to the maximal open intervals $I$ on which $d(h)$ is well defined and constant are open bands $B = B(I) \subset  S$,
$$
B = \bigcup_{h \in I} L(h).
$$
We denote by $d(B)$ the common latitude degree of $f$ on latitudes in $B$.  
Lemma~\ref{l.poletopole} implies that  the value of $\varphi $ at the endpoints of $I$ is $0$ or $1$.

\begin{Propn}
\label{p.globaldegree}
The global degree of $f$ is
$$
d(f) = \sum  \Delta _{I}\varphi  \cdot d(B)
$$
where $\Delta _{I}\varphi  = \varphi (b) - \varphi (a)$ and the sum is taken over   the bands and their band intervals  $I = (a,b)$. 
\end{Propn}

\begin{proof}[\bf Proof]
By continuity of $\varphi $ there are at most finitely many band  intervals $I = (a, b)$ with $\varphi (a) \not=  \varphi (b)$, so the sum makes sense.    The degree of $f$ is independent of homotopy, so we can deform  $f$ on each band $B$  in order  that $\varphi $ becomes linear on the band interval  $I$.   Then we can homotop $f$ further to make the intervals on which $\varphi $ is constant become points.  Finally we can homotop $f$ on each latitude so that up to homothety, it sends $z$ to $z^{d}$ or $\overline{z}^{d}$.   (In the case $d=0$ we homotop $f$ so that on each latitude, up to homothety it is the constant  map   $z \mapsto i$.)  The net effect is that by homotopy, we can assume    the latitude map of $f$   is a sawtooth as shown in Figure~\ref{f.sawtooth}, and the map on each latitude is the simplest possible.   

\begin{figure}[htbp]
\centering
\includegraphics[scale=.65]{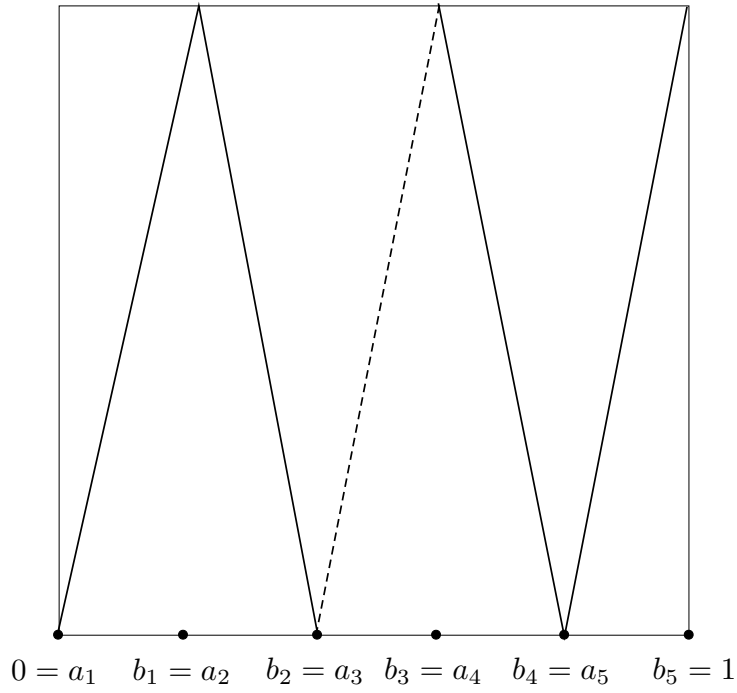}
\caption{The graph of a latitude map   with five bands.  Each pole is fixed  by $f$.  We imagine $f$ sending the first, second, fourth, and fifth band to $S$ with latitude degree $d_{1}, d_{2}, d_{4}, d_{5}$, and sending the third band to $S$ with latitude degree zero.}
\label{f.sawtooth}
\end{figure}

Take a regular value $v \in S$ of $f$    near $  (1/2, 0, 1/2)$.   The global degree of $f$ is the number of pre-images of $v$, counted with multiplicity.     There are no pre-images in bands with latitude degree zero, because those bands are sent to the half-longitude through $(0, 1/2, 1/2)$.   In a band $B$ with $\Delta _{I}\varphi  = 1$ and latitude degree $d > 0$ there are $d$ pre-images.  The same is true if $\Delta _{I}\varphi  = -1$ and $B$ has latitude degree $-d < 0$.  The other bands give pre-images with corresponding negative multiplicity, so the total number of pre-images, counted with multiplicity, is the sum $\sum \Delta _{I}\varphi  \cdot d(B)$ as claimed.
\end{proof}

If $\Delta _{I} \varphi  = 1$ we say that the corresponding band is \textbf{directed upward} or \textbf{ascends}, while if $\Delta _{I}\varphi  = -1$ it is \textbf{directed downward} or \textbf{descends}.  If $\Delta _{I}\varphi  = 0$, the band is \textbf{neutral}.

\begin{Lemm}
\label{l.directedbands}
There exist directed bands.
\end{Lemm}
\begin{proof}[\bf Proof]
Since $f$ is surjective, so is $\varphi $, and   $\varphi $ carries some minimal interval $[a,b] \subset  [0,1]$ onto $[0,1]$ with $\varphi (a) = 0$ and $\varphi (b) =1$ or vice versa.  The interval $(a,b)$  corresponds to a directed band.
\end{proof}
 
\begin{Lemm}
\label{l.invariantlatitude}
If the band $B$ is directed and $\abs{d(B)} \geq  2$   then $B$ contains  an $f$-invariant latitude.
\end{Lemm}

\begin{proof}[\bf Proof]

\underline{Case 1.}  $\partial B$ is the pair of poles, each being fixed by $f$.       Since $f$ has latitude degree $2$ or better  and is $C^{1}$ at the poles, its derivative there is zero.  For let $p$ be a pole.  The derivative of $f$ at $p$ is a linear map of the tangent space $T_{p}S$ to itself.  Infinitesimally it preserves the latitude foliation, so it is a scalar multiple of a rotation or reflection, $T_{p}f = cR$.  But if $c \not=  0$ then $f$ has latitude degree $\pm 1$ on $B$, contrary to the hypothesis that $\abs{d(B)} \geq  2$.  Thus, $c=0$ and the poles are sinks for $f$, so the latitude map $\varphi  : [0, 1] \rightarrow [0,1]$ has 
\begin{equation*}
\begin{split}
\varphi (0) & = 0 \quad \varphi ^{\prime}(0) = 0
\\
\varphi (1) &= 1 \quad \varphi ^{\prime}(1) = 0.
\end{split}
\end{equation*}
This gives a fixed point $h$ of $\varphi $ with $0 < h < 1$, and    $L(h)$ is invariant under $f$.  See Figure~\ref{f.phifixespoles}.

\underline{Case 2.}  $\partial B$ is the pair of poles, and $f$ interchanges them.    Differentiability of $f$ is irrelevant.   
See Figure~\ref{f.phifixespoles}.

\begin{figure}[htbp]
\centering
\includegraphics[scale=1.2]{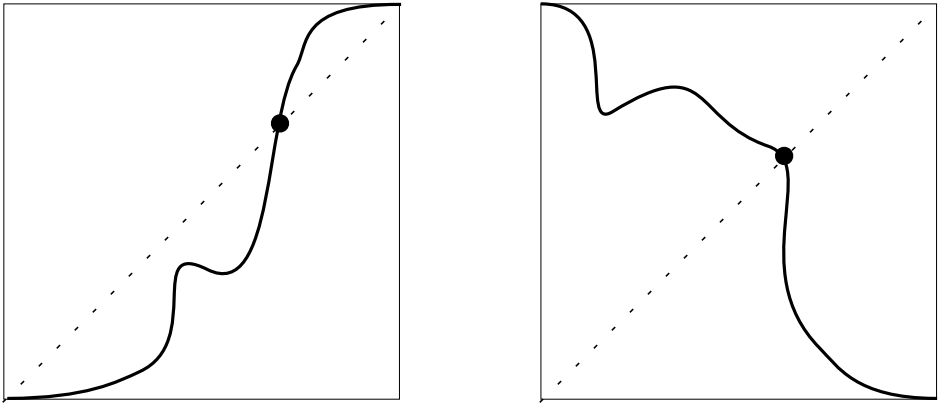}
\caption{The graph of $\varphi $. The first   corresponds to Case 1 and the second to Case~2.}
\label{f.phifixespoles}
\end{figure}

\underline{Case 3.}  $\partial B$ is a pole, fixed by $f$,  and a latitude.  The pole is a sink and the latitude is sent the other pole.  This gives a fixed point of $\varphi $ as in Case~1.

\underline{Case 4.} $\partial B$ is a pole and a latitude, and $f$ sends the pole to the opposite pole.  This gives a fixed point as in Case~2.

\underline{Case 5.}   $\partial B$ is two latitudes, say $L_{1}, L_{2}$ with heights $0 < h_{1} < h_{2} < 1$.    Then $\varphi  : [h_{1}, h_{2}] \rightarrow  [0,1]$ sends $[h_{1}, h_{2}]$ over itself, so it has a fixed point $h \in (h_{1}, h_{2})$.  The latitude $L(h)$ is invariant by $f$.   
\end{proof}

\begin{Rmk}
The   hypothesis that $f$ is $C^{1}$ is used only in Cases 1 and 3.  
\end{Rmk}

\begin{Propn}
\label{p.2tothen}
If the band $B$ is directed and $\abs{d(B)} \geq  2$       then $f^{n}$ has at least $2^{n}$ fixed points.
\end{Propn}

\begin{proof}[\bf Proof]
By Lemma~\ref{l.invariantlatitude} there is an invariant latitude on which $f$ has degree $k$ with  $\abs{k} \geq 2$.  The  $n^{\textrm{th}}$ iterate of such a map of the circle has at least $\abs{k}^{n}$ fixed points.
\end{proof}

\section{Counting Fixed Points}
\label{s.counting}

\begin{Lemm}
\label{l.degreeproduct}
If $f$ and $g$ preserve the latitude foliation then the latitude degree of $f \circ g$ is at most the product of their latitude degrees.
\end{Lemm}
\begin{proof}[\bf Proof]
This is a standard fact about   maps of the circle.
\end{proof}

\begin{Lemm}
\label{l.numberoffixedpoints}   Let $f : S \rightarrow  S$ be a continuous latitude preserving surjection, not necessarily of degree $2$.  
If each of its directed bands $B$ has $\abs{d(B)} \leq  1$ then $f$ has more than $d(f)$   fixed points.
\end{Lemm}

\begin{proof}[\bf Proof]  We count the directed bands as follows. 
\begin{itemize}

\item[(a)]
  $a$ is the number of ascending bands with $d(B) = 1$.
\item[(b)]
    $b$ is the number of ascending bands   with $d(B) = -1$.
\item[(c)]
 $c$ is the number of descending  bands with $d(B) = 1$.

\item[(d)]
 $d$ is the number of descending bands with $d(B) = -1$.
 
 \item[(e)]  
 $e$ is the number of directed bands with $d(B) = 0$.

\end{itemize}

We think of the graph of $\varphi $ as composed of ``legs'' that join $[0,1] \times 0$ to $[0,1] \times  1$.  Formally, they are arcs in the open square $(0,1) \times  (0,1)$, and they   correspond to the bands $B_{1}, \dots  , B_{N}$.  See Figures~\ref{f.polesswitch}, \ref{f.polestoSouthpole}, and \ref{f.polesfixed}. Ascending legs correspond to ascending bands, descending to descending.  We call a leg corresponding to a band with $d(B) = -1$ a \textbf{reverse-leg}, and we call a leg corresponding to a band with $d(B) = 0$ a \textbf{zero-leg}.  There are $b+d$ reverse-legs and $e$ zero-legs.  Each intersection of the diagonal $\Delta $ with a reverse leg produces two fixed points of $f$, since such an intersection gives an $f$-invariant latitude $L$, and $f : L \rightarrow  L$ reverses   orientation.  Each intersection of $\Delta $ with a zero-leg produces at least one fixed point of $f$, since such an intersection gives an $f$-invariant latitude $L$, and $f : L \rightarrow  L$ has degree zero.  Intersections of $\Delta $ with other legs need not produce fixed points.

The rest of the proof is a counting argument in which there are   three cases to consider, concerning  how $f$ affects the poles, and twelve subcases concerning   which legs     $\Delta $ crosses.

Let $p$ be the number of fixed poles, and $P = N_{1}(f)$ the number of fixed points of $f$ (including the fixed poles). 
According to Proposition~\ref{p.globaldegree} the  degree of $f$ is $d(f) = (a+d)-(b+c)$.  We are trying to show that $P > d(f)$.  Naively, we imagine $\Delta $ crossing all the legs, and hence generating corresponding fixed points  --  two for each (b)-crossing, two for each (d)-crossing, and one for each (e)-crossing.  Thus we would    hope  
$$
P \;\; \geq \;\;  p + e + 2(b+d) \;\; >\;\; (a+d) - (b+c), 
$$
which leads us to ask whether $p + N >  2a -2b$ since $ a+b+c+d + e = N$.  This inequality is in fact true, but we need a stronger one  because if $\Delta $ fails to cross some legs of type (b), (d), or (e) then we would have over-estimated the fixed points.  We write
\begin{equation*}
\begin{split}
P - d(f)  \;\; &\geq   \;\; p + e + 2(b+ d) - (a+d) + (b+c) -r
\\
&= p + N -2a + 2b -r
\end{split}
\end{equation*}
where $r$ is the correction term due to $\Delta $ missing legs of type (b), (d), (e).  Thus,
$$
r  \, = \,
\begin{cases}
   0   &  \textrm{ if } \Delta \textrm{ crosses all legs,}
  \\
   1  &  \textrm{ if }  \Delta \textrm{ crosses all legs except one of type (e),}
  \\
   2  &  \textrm{ if }  \Delta \textrm{ crosses all legs except two of type (e),}
     \\
   2  &  \textrm{ if }  \Delta \textrm{ crosses all legs except one of type (b) or (d),}
     \\
   3  &  \textrm{ if }  \Delta \textrm{ crosses all legs except one of type (b) or (d)} 
   \\
   &\textrm{ and one of type (e),}
     \\
   4  &  \textrm{ if }  \Delta \textrm{ crosses all legs except two of type (b) or (d).}
\end{cases}
$$

\underline{Case 1.} $\varphi $ switches the poles.  Then $p=0$,  $\varphi (0) = 1$, $\varphi (1) = 0$, $N$ is odd, and there are $(N-1)/2$ ascending legs.  In particular $a$ is at most the number of ascending legs, so   $2a \leq  N-1$.  The diagonal meets all the legs so we have $r=0$.  This gives   
\begin{equation*}
\begin{split}
P - d(f)  \;\; &\geq   \;\; p + N -2a + 2b - r
\\
&\geq  \;\;0 + N -(N-1) + 2b - 0  \;\; \geq \;\; 1
\end{split}
\end{equation*}
 since  $2b \geq  0$.  Thus $P> d(f)$ in Case 1. See Figure~\ref{f.polesswitch}.
  \begin{figure}[htbp]
\centering
\includegraphics[scale=.55]{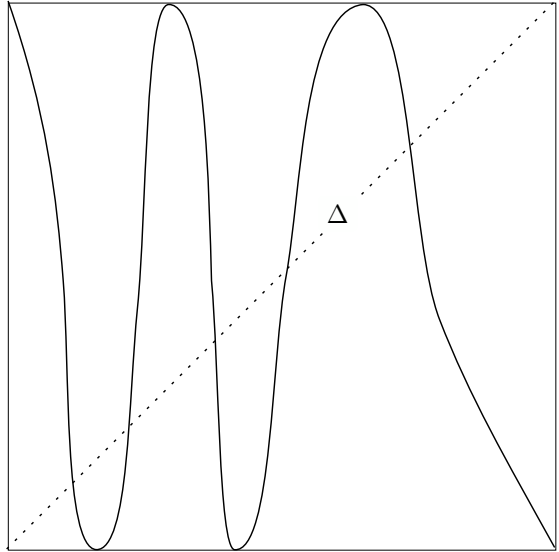}
\caption{The graph of a latitude map $\varphi $ when $f$ interchanges the poles. The diagonal $\Delta $ meets each leg.  The graph is simplified  --  each leg can actually intersect $\Delta $ several times;  in the case shown, the graph  must intersect $\Delta $ five times.  Also unshown are  possible arcs in the graph that join the top or bottom of the square to itself.  They correspond to neutral bands}
\label{f.polesswitch}
\end{figure}
  
  \underline{Case  2.}  $f$ sends both poles to the same pole, say the Southpole.  Then $p=1$, $\varphi (0) = 0 = \varphi (1)$, $N$ is even,   and there are $N/2$ ascending legs.  In particular, $2a \leq  N$.  The diagonal crosses all the legs except possibly the first.  See Figure~\ref{f.polestoSouthpole}
  \begin{figure}[htbp]
\centering
\includegraphics[scale=.55]{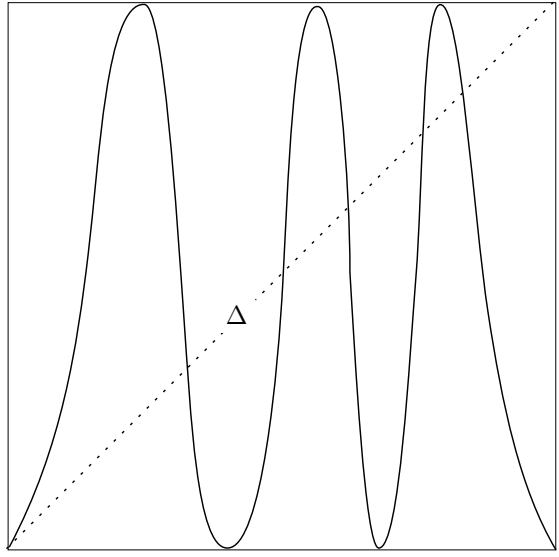}
\caption{The graph of a latitude map when $f$ sends both poles to the Southpole.  It is also simplified.  The diagonal meets each leg except perhaps the first.}
\label{f.polestoSouthpole}
\end{figure}
  
  \underline{Case 2a.} The first leg is of type (a).  Then $\Delta $ crosses all the legs of type (b), (d), (e), which implies $r=0$.  This gives    
    \begin{equation*}
\begin{split}
P - d(f)  \;\; &\geq   \;\; p + N -2a + 2b - r
\\
&\geq  \;\;1 + N -N + 2b -0  \;\; \geq \;\; 1
\end{split}
\end{equation*}
 since  $2b \geq  0$.

  \underline{Case 2e.}  The first leg is of type (e). Then $r=1$.  Also,    $2a \leq N-2$ since there are $N/2$ ascending legs, one of which (the first one) is of type (e), not of type (a).  This gives 
        \begin{equation*}
\begin{split}
P - d(f)  \;\; &\geq   \;\; p + N -2a + 2b - r
\\
&\geq  \;\;1 + N -(N-2) + 2b -1 \;\; \geq \;\; 2
\end{split}
\end{equation*}
 since $2b \geq 0$.  
  
    \underline{Case 2b.}   The first leg is of type (b).  Then $r=2$, and as in the previous case, $2a \leq N-2$. This gives
        \begin{equation*}
\begin{split}
P - d(f)  \;\; &\geq   \;\; p + N -2a + 2b - r
\\
&\geq  \;\;1 + N -(N-2) + 2b -2  \;\; \geq \;\; 3
\end{split}
\end{equation*}
since $2b \geq 2$.  Thus $P > d(f)$ in Case 2.

 \underline{Case 3.}  $f$ fixes both poles.  Then $p=2$, $\varphi (0) = 0$, $\varphi (1) = 1$, $N$ is odd, and there are $(N+1)/2$ ascending legs.  In particular, $2a \leq  N+1$.  The diagonal crosses all legs except possibly the first and last.  See Figure~\ref{f.polesfixed}.
 
 \begin{figure}[htbp]
\centering
\includegraphics[scale=.55]{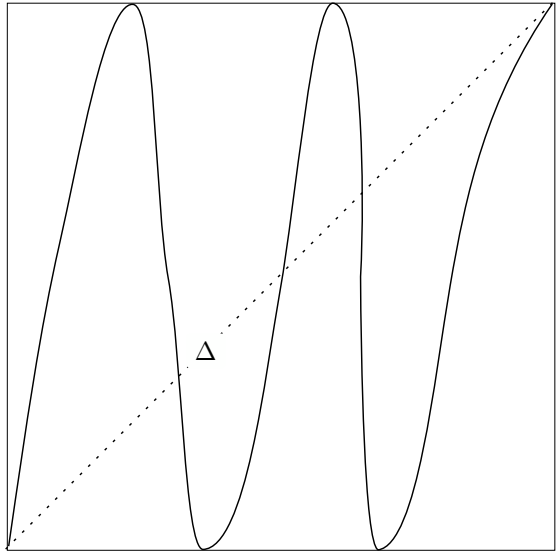}
\caption{The graph of a latitude map when $f$ leaves the poles fixed.  It is also simplified.  The diagonal meets each leg except perhaps the first and last.}
\label{f.polesfixed}
\end{figure}
 
 \underline{Case 3aa.}  The first and last legs are of type (a).  Then $\Delta $ crosses all the legs of type (b), (d), (e), so $r=0$.  This gives     
         \begin{equation*}
\begin{split}
P - d(f)  \;\; &\geq   \;\; p + N -2a + 2b - r
\\
&\geq  \;\; 2 + N - (N+1) + 2b - 0 \;\; \geq \;\; 1
\end{split}
\end{equation*}
 since  $2b \geq 0$. 
  
 \underline{Case 3ae.}  The first leg is of type (a) and the last  of type (e).  Then $r=1$.  Also, since the first and last legs ascend, and since one of them (the last one) is of type (e), not of type (a), we have   $2a \leq N-1$.  This gives 
         \begin{equation*}
\begin{split}
P - d(f)  \;\; &\geq   \;\; p + N -2a + 2b - r
\\
&\geq  \;\;2 + N -(N-1) + 2b -1  \;\; \geq \;\; 2
\end{split}
\end{equation*}
since  $2b \geq 0$. 
  
\underline{Case 3ab.}  The first leg is of type (a) and the last of type (b). Then $r=2$ and, as in the previous case, $2a \leq  N-1$.  This gives
        \begin{equation*}
\begin{split}
P - d(f)  \;\; &\geq   \;\; p + N -2a + 2b - r
\\
&\geq  \;\;2 + N -(N-1) + 2b -2 \;\; \geq \;\; 3
\end{split}
\end{equation*}
since  $2b \geq 2$.  
  
\underline{Case 3ea.}  The first leg is of type (e) and the last of type (a). This   is symmetric to  Case~3.ae. 
  
\underline{Case 3ee.}  The first  and last legs are of type (e). If $N \geq  3$ then  $r=2$ and $2a \leq N-3$ since two ascending legs are not of type (a).  This gives 
        \begin{equation*}
\begin{split}
P - d(f)  \;\; &\geq   \;\; p + N -2a + 2b - r
\\
&\geq  \;\;2 + N -(N-3) + 2b -2  \;\; \geq  \;\; 3
\end{split}
\end{equation*}
since $2b \geq 0$.  On the other hand, if $N = 1$ then there is just one ascending leg, and it is of type (e).  Then $r=1$ and $a=0=b$.  This gives
       \begin{equation*}
\begin{split}
P - d(f)  \;\; &\geq   \;\; p + N -2a + 2b - r
\\
&=  \;\;2 + 1 -0 + 0 -1  \;\; =  \;\; 2.
\end{split}
\end{equation*}

\underline{Case 3eb.}  The first leg is of type (e) and the last of type (b).  Then $N \geq  3$,  $r=3$, and $2a \leq  N-3$.  This gives
        \begin{equation*}
\begin{split}
P - d(f)  \;\; &\geq   \;\; p + N -2a + 2b - r
\\
&\geq  \;\;2 + N -(N-3) + 2b -3 \;\; \geq  \;\; 4
\end{split}
\end{equation*}
since $2b \geq 2$. 
  
\underline{Case 3ba.}  The first leg is of type (b) and the last of type (a).  This is symmetric to Case~3.ab.
  
\underline{Case 3be.}  The first leg is of type (b) and the last of type (e).  This is symmetric to Case~3.eb.
  
\underline{Case 3bb.}  The first  and last legs are of type (b).   If $N \geq  3$ then  $r=4$ and $2a \leq N-3$. This gives
        \begin{equation*}
\begin{split}
P - d(f)  \;\; &\geq   \;\; p + N -2a + 2b - r
\\
&\geq  \;\;2 + N -(N-3) + 2b -4 \;\; \geq  \;\; 5
\end{split}
\end{equation*}
since $2b \geq 4$.   On the other hand, if $N=1$ then there is just one ascending leg and it is of type (b).  Then $r=2$, $a=0$, and $b =1$.  This gives
       \begin{equation*}
\begin{split}
P - d(f)  \;\; &\geq   \;\; p + N -2a + 2b - r
\\
&=  \;\;2 + 1 -0 + 2 -2 \;\; =  \;\; 3.
\end{split}
\end{equation*}
Thus $P > d(f)$ in Case 3, which completes the proof of Lemma~\ref{l.numberoffixedpoints}.
\end{proof}

\begin{Rmk}
The degree of $f$ can be negative, in which case Lemma~\ref{l.numberoffixedpoints} says nothing.  One can re-work the preceding estimates to show that
$$
P - \abs{d(f)} \geq -1,
$$
but we do not need this.
\end{Rmk}

\begin{proof}[\bf Proof of   Theorem~\ref{t.fixedpoints}]  We assume that the $C^{1}$ latitude preserving map $f : S \rightarrow  S$ has degree $2$ and claim that $f^{n}$  has at least $2^{n}$ fixed points.  
By Lemma~\ref{l.directedbands} there exist directed bands.  
If there is a directed band $B$ with $\abs{d(B)} \geq 2$ then the result follows from Proposition~\ref{p.2tothen}.  If all the directed bands have $\abs{d(B)} \leq  1$ then by Lemma~\ref{l.degreeproduct} the same is true of $f^{n}$.  Applying  Lemma~\ref{l.numberoffixedpoints} to $f^{n}$, we conclude that $f^{n}$ has more than $d(f^{n}) = 2^{n}$ fixed points.
\end{proof}

\begin{Rmk}
The proof of Corollary~\ref{c.GRF} does not require the full force of Lemma~\ref{l.numberoffixedpoints}.  It suffices that $P = N_{1}(f)  \geq d(f) - k$  where $k$ is an absolute constant.  For, as just observed, if all the bands for $f$ have $\abs{d(B)} \leq  1$ then the same is true for $f^{n}$.  Thus, $N_{1}(f^{n}) \geq d(f^{n}) - k$.  Since $N_{1}(f^{n}) = N_{n}(f)$, we get 
$$
N_{n} \geq 2^{n} - k
$$
from which Corollary~\ref{c.GRF} is immediate.  

Here is the caseless proof of this weaker inequality.  Since all the bands of $f$ have $\abs{d(B)} \leq  1$, the  number of fixed points of $f$ is at least
$ 
e-2 + 2(b+d-2)
$.
For the graph of the latitude map  has at least $e-2$ legs of type (e)  that cross the diagonal,  and at least $b+d-2$ legs of type (b) or (d) that cross the diagonal.  The former produce one fixed point each, and the latter two fixed points each.  
For it is only the first and last legs of the latitude map that can fail to intersect the diagonal.  This quantity minus the degree of $f$ is
$ N -2a + 2b - 7$ where $N$ is the number of legs.  That is,
$$
N_{1}(f) - d(f) \geq N-2a + 2b - 7 \geq -7
$$
since $2a \leq  N + 1$ and $2b \geq  0$.  
\end{Rmk}

\section{Three Examples}
\label{s.examples}
 
It is possible to code a latitude preserving map $f$  as follows.  Take any finite sequence of integers $(n_{0}; d_{1}, \dots , d_{N})$ such that $n_{0}$ is $0$ or $1$, and make the following interpretation.  If $n_{0}= 0$ and $N$ is even then consider directed bands $B_{1}, \dots , B_{N}$ such that
\begin{equation*}
\begin{split}
B_{1} & \textrm{ ascends and has latitude degree $d_{1}$}
\\
B_{2} & \textrm{ descends and has latitude degree $d_{2}$}
\\
B_{3} & \textrm{ ascends and has latitude degree $d_{3}$}
\\
\dots  & \dots 
\\
B_{N} & \textrm{ descends and has latitude degree $d_{N}$.}
\end{split}
\end{equation*}
On the other hand, if $N$ is odd then the last band ascends and has latitude degree $d_{N}$.  Similarly, if $n_{0} = 1$ then every ascending band becomes descending and vice versa.  The latitude degrees remain the same.  The choice $n_{0}= 0$ indicates that $f$ fixes the Southpole, while $n_{0} = 1$ indicates that $f$ sends the Southpole to the Northpole.

Since ascending and descending bands alternate, such a code is well defined for any latitude preserving map $f : S \rightarrow  S$, and it describes $f$ well up to non-monotonicity of the latitude maps $\varphi  : I \rightarrow  (0,1)$,   the presence of neutral bands, and latitude rotation.

The global degree of $f$ is the alternating sum $d_{1} - d_{2} + \dots  + (-1)^{N+1}d_{N}$.

Assume that $\abs{d_{i}} \leq  1$ for $1 \leq  i \leq  1$.
It is easy to see that to each code $(n_{0}; d_{1}, \dots , d_{N})$ there corresponds a    smooth  endomorphism $f$ with the following properties. 
\begin{itemize}

\item
The map $f$ preserves the latitude foliation, and up to homothety it is $z \mapsto z$ or $z \mapsto  \overline{z}$ on each latitude.  
\item
$S$ has   $N$ bands $B_{i}$ with latitude degrees $d_{i}$.
\item
If $d_{i} \not=  0$ then $f$ sends $B_{i}$ diffeomorphically to the sphere minus the poles.
\item
If $d_{i} =  0$ then $f$ sends $B_{i}$   to the prime meridian $M$.  ($M$ is the longitude arc  that joins the poles and contains the point  $(1/2, 0, 1/2)$.)

\end{itemize}            
We refer to such an $f$ as a good representative of the code.  It is not unique.   Here are three examples.
\begin{enumerate}

\item
The code is $(0;1,1)$.  Let $f$ be a   good representative of the code. Then $S$ has two bands, the Southern and Northern hemispheres, minus the poles and equator. 
\begin{itemize}

\item
$f$ wraps the Southern hemisphere    upward over $S$, pinching the equator to the Northpole, and preserving    the latitude orientation.  It    fixes the Southpole.
\item
$f$ wraps the Northern hemisphere    downward over $S$, pinching the equator to the Northpole, and preserving  the   latitude orientation.  It   sends the Northpole to the Southpole.

\end{itemize} 

  The map $f$ has degree zero, preserves the latitude foliation,  and   fixes the Southpole.  Its latitude map $\varphi $ is unimodal, so its   entropy is $\ln 2$, and since $f$ fibers over $\varphi $ with diffeomorphisms in the circle fibers, it too has entropy $\ln 2$.    Now take a latitude-preserving rotation $R$ of the sphere by an angle $\theta $ where $\theta /2\pi $ is irrational.  Set $F = f \circ  R$.    The entropy is unaffected and   $F$ preserves the latitude foliation.   The only fixed point of $F^{n}$ is the Southpole, for the effect of $F^{n}$ on any invariant latitude    is an irrational rotation.  Thus, the Growth Rate Inequality holds for $F$, in the form $0 \geq   0$.

\item
The code is $(0;1,-1)$.  Let $g$ be a good representative of the code. Again, $S$ has the hemisphere bands.

\begin{itemize}
\item
$g$ wraps the Southern hemisphere    upward over $S$, pinching the equator to the Northpole, and preserving    the latitude orientation. It     fixes the Southpole.

\item
$g$ wraps the Northern hemisphere    downward over $S$, pinching the equator to the Northpole, and reversing    the latitude orientation. It   sends the Northpole to the Southpole.

\end{itemize}

The map $g$ has degree $2$ and fixes the Southpole.  Again, let $R$ be an irrational rotation of the sphere and set $G = g \circ  R$.  The map $G$ preserves the latitude foliation.  Its latitude map $\varphi $ is unimodal, so its   entropy is $\ln 2$, and since $f$ fibers over $\varphi $ with diffeomorphisms in the circle fibers, it too has entropy $\ln 2$.  The map      $\varphi ^{n}$ has $2^{n}$ fixed points. Each corresponds to a $G^{n}$-invariant latitude $L$.    On half of them, $G^{n}$ preserves the latitude orientation, and on half of them $G^{n}$  reverses it.  On the latitudes with preserved orientation, $G^{n}$ is an irrational rotation and has no periodic points.  On the latitudes with reversed orientation $G^{n}$ has two fixed points.  Altogether,  $G^{n}$ has $2^{n}$ fixed points, so the Growth Rate Inequality holds for $G$, in the form $\ln 2  \geq   \ln 2$.

\item
The code is $(0;1,0,1)$.  Let $h$ be a good representative of the code.    Then $S$ has three bands $B_{1}, B_{2}, B_{3}$.  
\begin{itemize}

\item
 $h$ wraps $B_{1}$  upward over $S$, pinching its  boundary latitude to the Northpole, and preserving   the latitude orientation.  It  fixes the Southpole.
\item
$h$ wraps $B_{2}$  downward along the prime meridian $M$, pinching its lower boundary latitude to the Northpole and its upper boundary latitude to the Southpole.    The $h$-image of $B_{2}$ equals $M$.
\item
$h$ wraps $B_{3}$  upward over $S$, pinching its boundary latitude to the Southpole, and preserving  the latitude orientation. It  fixes the Northpole.

\end{itemize}  

The map $h$ has degree $2$ and fixes both poles.  Its latitude map $\varphi $ is bimodal, so its   entropy is $\ln 3$, and since $h$ fibers over $\varphi $ with diffeomorphisms in the circle fibers, it too has entropy $\ln 3$.    Take again an irrational rotation of the sphere, $R$,  and set $H = h \circ R$.  The map $\varphi ^{n}$ has $3^{n}$ fixed points.  For the majority of them, their $\varphi $-orbits contain    points in $I_{2}$. (That is, the proportion of the $\varphi ^{n}$-fixed-points whose orbit includes points of $I_{2}$ tends to $1$ as $ n\rightarrow \infty $.  The other orbits lie in a zero measure Cantor set.)    For each such fixed point of $\varphi ^{n}$ we have an $H^{n}$-invariant latitude $L$, at least one of whose $H$-iterates lies in $B_{2}$, and therefore $H^{n}(L) $ is the single point $L \cap M$.  The other fixed points of $\varphi ^{n}$ correspond to $H^{n}$-invariant latitudes whose $H$-orbits avoid $B_{2}$. On these latitudes, $H^{n}$ is an irrational rotation, so they give no fixed points.   Altogether, $H^{n}$ has $3^{n}$ fixed points, nearly each of which contributes one fixed point of $H^{n}$, so the Growth Rate Inequality holds in the form $\ln 3 \geq  \ln 2$.  
\end{enumerate}

\begin{Rmk}
It is possible that Lemma~\ref{l.numberoffixedpoints} and our theorem have   proofs using a coding like this.  We would need to know how the coding is affected when we take a banding $B_{1},\dots  , B_{N}$ of $S$  and consider a sub-banding of each $B_{i}$ as
$$
B_{ij} = B_{i} \cap f^{-1}(B_{j}).
$$
\end{Rmk}

\end{document}